\documentclass[12pt, a4paper]{article}

\usepackage{amsmath, amsthm, amsfonts, amssymb}
\usepackage{setspace}
\usepackage{fullpage}
\usepackage{enumitem}
\usepackage{graphicx}
\usepackage{floatpag}
\usepackage[dvipsnames]{xcolor}
\usepackage[initials]{amsrefs} 
\usepackage{float}
\usepackage{caption}
\usepackage{subcaption} 
\usepackage[T1]{fontenc}
\usepackage{hyperref}
\usepackage{cleveref}
\usepackage{comment}

\hypersetup{colorlinks=true,
	citecolor=blue,
	filecolor=blue,
	linkcolor=blue,
	urlcolor=blue
}

\newtheorem{theorem}{Theorem}

\newtheorem{proposition}[theorem]{Proposition}

\newcommand{\F}{\mathbb{F}}
\newcommand{\Z}{\mathbb{Z}}

\title{On hypercube statistics}
	\date{\vspace{-5ex}}
	\author{ Noga Alon 
\thanks{Princeton University, USA; email:\texttt{ nalon@math.princeton.edu}}
			~~Maria Axenovich
		\thanks{Karlsruhe Institute of Technology, Germany;
			email:
			\mbox{\texttt{ maria.aksenovich@kit.edu}}}				
~~John Goldwasser			
\thanks{West Virginia University, USA;  email:\mbox{\texttt{ jgoldwas@math.wvu.edu}}}
}

\begin{document}	
			
\maketitle
	
\abstract{Let  $d \geq 1$ and $s \leq 2^d$ be nonnegative integers. 
For a subset $A$ of vertices of the  hypercube $Q_n$ and $n\geq d$, let 
$\lambda(n,d,s,A)$ denote the fraction of subcubes $Q_d$ of
$Q_n$ that contain exactly $s$ vertices of $A$. Let $\lambda(n,d,s)$
denote the maximum possible value of $\lambda(n,d,s,A)$ as
$A$ ranges over all subsets of vertices of $Q_n$, and let
$\lambda(d,s)$ denote the limit of this quantity as $n$ tends
to infinity.   
We prove several lower and upper bounds on $\lambda(d,s)$, showing that
for all admissible values of $d$ and $s$ it is larger than $0.28$.
We also show that the
values of $s=s(d)$ such that $\lambda(d,s)=1$ are exactly
$\{0,2^{d-1},2^d\}$.
In addition we  prove  that  if $0<s< d/8$, then 
$\lambda(d, s) \leq 1 -   \Omega(1/s)$, and that if $s$ is divisible by a power
of $2$ which is $\Omega(s)$ then $\lambda(d,s) \geq 1-O(1/s)$.
We suspect that $\lambda(d,1)=(1+o(1))/e$ where the $o(1)$-term tends to
$0$ as $d$ tends to infinity, but this remains open, as does the problem of
obtaining tight bounds for essentially all other quantities $\lambda(d,s)$.
}	
	
\section{Introduction}	

Let $Q_n$ be the hypercube of dimension $n$ whose vertices are identified with $n$-component  binary vectors.
For a subset $A$ of vertices of $Q_n$ and $d\leq n$, let 
$\lambda(n,d,s,A)$ denote the fraction of subcubes $Q_d$ of
$Q_n$ that contain exactly $s$ vertices of $A$. Let $\lambda(n,d,s)$
denote the maximum possible value of $\lambda(n,d,s,A)$ as
$A$ ranges over all subsets of vertices of $Q_n$, and let
$\lambda(d,s)$ denote the limit of this quantity as $n$ tends
to infinity. It is easy to see that 
the limit exists, and is the infimum over $n$ of
$\lambda(n,d,s)$ as for any fixed $d,s$ the function
$\lambda(n,d,s)$ is monotone non-increasing in $n$.\\

The problem of determining or estimating the quantities
$\lambda(n,d,s)$ and $\lambda(d,s)$ is motivated by
the questions and  results of Goldwasser and Hansen 
on counting structural configurations in hypercubes \cite{GH},
as well as by the results on edge-statistics in graphs by 
Alon, Hefetz, Krivelevich, and Tyomkyn \cite{AHKT},    
Kwan,  Sudakov, and  Tran \cite{KST},  
Martinsson, Mousset, Noever, and 
Trujic \cite{MMNT}, and Fox and Sauermann \cite{FS}.\\
	
Clearly $\lambda(d,s)=\lambda(d,2^d-s)$ and $\lambda(d,0)=1$. In addition, if $s=2^{d-1}$, then we see that $\lambda(d,s)=1$ by taking all vertices of the hypercube with even number of ones. 
To state our results,  define the generalized Johnson's 
Graph $J(4s, 2s, s)$  whose vertex set is the set of $2s$-element 
subsets of a  $4s$-element set, in which two vertices are 
adjacent if  and only if the corresponding sets intersect in exactly 
$s$ elements. Let $\omega(s)= \omega(J(4s, 2s, s))$ denote the 
clique number of $J(4s, 2s, s)$. It is easy and known that 
$\omega(s) \leq 4s-1$ with equality if and only if a Hadamard 
matrix of order $4s$ exists, see for example Godsil and Royle \cite{GR}.
We first state our upper bounds on $\lambda(d,s)$. We use the  
notation $t(n,k)$ for the number of edges in the Tur\'an graph $T(n,k)$, 
that is, the complete $k$-partite $n$-vertex graph with parts that are 
as equal as possible.
Denote the density $t(n,k)/\binom{n}{2}$ by $\pi(n,k)$.
	
\begin{theorem}
	\label{ub}
Let $s$ and $d$ be integers. Then  $\lambda(d,s)=1$ 
	if and only if $s\in \{0, 2^d, 2^{d-1}\}$.  If $1< s <2^{d-1}$, 
then 
$$\lambda(d,s) \leq \lambda(d+2, d,s) = \pi(d+2, \omega(s)) \leq 
\left(1 - \frac{1}{4s-1}\right)\left(1+ \frac{1}{d+1}\right).$$
In particular, $\lambda(d+2, d,s)=1$ iff $d+2\leq \omega(s)$.  
	When $s=1$, we have   $\lambda(d,1) \leq \lambda(d+2, d, 1) 
	= \pi(d+2, 3) $ for $d< 6$, and $\lambda(d+2, d,1)= 3/4$ otherwise.
\end{theorem}

Note that the general upper bound implies in particular that if  $s$ is not large, say, $0<s< d/8$, then 
$\lambda(d, s) \leq 1 -   \Omega(1/s)$.

Next we provide the lower bounds that are described probabilistically. 
Let $c_d$ denote the probability that a random $d$ by $d$ binary
matrix whose rows are random independent non-zero vectors of $\F_2^d$ is 
nonsingular (in $\F_2$). It is easy and well known that 
$c_d=\prod_{i=1}^{d-1} (1-\frac{2^i-1}{2^d-1})$, which is roughly 
$0.289$ for large $d$.\\

For $1 \leq k \leq d$, let $c(d,k)$  denote the probability that
a random $(d-k)$ by $d$ binary matrix whose columns are uniform
random vectors in $\F_2^{d-k}$ is of rank $d-k$ (over $\F_2$). 
It is a bit better to take here too only nonzero random column vectors,
but to simplify the computation we consider this slightly
suboptimal version. By choosing the rows (not the columns) 
of the matrix one by one ensuring that each row does not lie
in the span of the previous ones it is easy to see that
$$
c(d,k)=\prod_{i=0}^{d-k-1} \left(1-\frac{2^i}{2^d}\right).
$$ 
Note that this quantity is larger than $1-\frac{1}{2^k}$. 

The first simple lower bound in the theorem below 
appears in the recent paper Goldwasser and Hansen \cite{GH}, 
we include the proof here for completeness. Note that this 
lower bound approaches $e^{-1}\approx 0.37$ as $d$ tends to infinity. 
	
\begin{theorem}\label{lb}
For any integer $d\geq 2$, $ \lambda(d,1) \geq \left( 1- 2^{-d} \right)^{2^d -1}$.
For all admissible $d$ and $s$,
$\lambda(d,s) \geq c_d$. Moreover,  for every $s$ of the form
$s=2^k \cdot j$, where $j$ is an odd integer,  which satisfies $0<s \leq 2^{d-1}$,
$\lambda(d,s) \geq c(d,k)$. In particular, for any $s$ which
is a power of $2$, $\lambda(d,s) \geq 1-\frac{1}{s}$.
\end{theorem}

\noindent	
{\bf Remark.}  Here we summarise the best bounds on  $\lambda(d,1)$  when $d = 2,3,$ or $4$.  
We have that $\lambda(2,1) \geq c_2=2/3$ from Theorem \ref{lb}.
Observe that $\lambda(d, 1) \geq 2/(d+1)$ by the following construction. The Hamming weight of a binary vector is its number of $1$'s.    For a fixed $d$, let $A$ be the set of all vertices in $Q_n$ with Hamming weight divisible by $d+1$.  A copy of $Q_d$-cube contains precisely one vertex in $A$ if and only if the smallest Hamming weight of any of its vertices is congruent to $0$ or $1 \pmod {d+1}$.
 Together with upper bounds established by Baber \cite{Ba}  using the Flag Algebra method, we have the following estimates for $d=2,3$, and $4$:
$2/3 \leq \lambda(2,1)\leq 0.68572$, 
$0.5 \leq \lambda(3,1)\leq 0.61005$, and 
$0.4 \leq \lambda(4,1)\leq 0.60254$. 	\\

The proofs of the main results are given in Section 
\ref{sec:main-proofs}. Section \ref{sec:number} contains some simple
number theoretic consequences.  In Section \ref{sec:approx} 
we consider an approximate version of the problem.
The final Section \ref{sec:conclusion} contains some 
concluding remarks and open problems.  

Throughout this note we call each of the $2^{n-d} {n \choose d}$
$d$-dimensional subcubes of $Q_n$ a
{\it $d$-cube} or a {\it copy of} $Q_d$. 
The $k-$th {\it layer} in the hypercube is the set of all vertices of 
Hamming weight $k$. 

\section{Proofs of the main results}\label{sec:main-proofs}

\begin{proof}[Proof of Theorem \ref{ub}]

	As already mentioned it is clear that $\lambda(d,s)=1$
	for $s \in \{0,2^{d-1},2^d\}$.
We first show the converse: if 
$\lambda(d,s)=1$ then $s\in \{0, 2^d, 2^{d-1}\}$.\\

Consider  a prime $p>2^d$, suppose $\lambda(d,s)=1$ and take a huge $n\gg p$ and
a subset $A$ of vertices of $Q_n$ so that every copy of $Q_d$ contains 
	exactly $s$ vertices
of $A$.
By a simple iterated application of the hypergraph Ramsey theorem, 
	there is a copy  $Q$  of $Q_p$  
in which every layer is either fully contained in $A$ or contains 
	no vertices of $A$, see for  example the Layered Lemma in \cite{AW}.
We have that  $Q$ has exactly $2^{p-d} s$ vertices of $A$ (as it consists
of $2^{p-d}$ pairwise disjoint copies of $Q_d$). This implies
that there is a subset $T$ of $\{0,1,2,\ldots, p\}$ so that
\begin{equation}\label{star}
s 2^{p-d}=\sum_{i \in T} \binom{p}{i}.
\end{equation}
If $T$ is empty then $s=0$, so assume $T$ is nonempty.  
	Consider three possible cases based on whether $0$ and/or 
	$p$ are in $T$.
 If neither $0$ nor $p$ are in $T$,  then the right hand side 
	of (\ref{star}) is divisible by $p$,
     which is impossible as the left hand side is not.
If both $0$ and $p$ are in $T$, then the right hand side 
	of (\ref{star})  is $2 \pmod p$. By Fermat's
little Theorem in this case 
	$s 2^{p-d}=2 \pmod p=2^p \pmod p$,  
	so $s = 2^d \pmod p$ and as $p>2^d$ and $p>s$ this gives
$s=2^d$.
Finally, if exactly one of $\{0,p\}$ is in $T$, then the right 
	hand side of (\ref{star}) is $1 \pmod p$, so in this  case
$s 2^{p-d}=1 \pmod p = 2^{p-1} \pmod p$ and thus $s= 2^{d-1} \pmod p$ 
	so $s=2^{d-1}$,
completing the proof of the first part of the theorem.\\~\\

For general upper bounds, we first consider  $\lambda(d+2, d, s)$.  
	The result will then follow by averaging.
Note that each $d$-cube in $Q_{d+2}$ can be uniquely described  
by the binary vectors that 
have prescribed values in some two positions, called {\it fixed positions} and running through all possible $2^d$ binary vectors on the remaining {\it variable positions}. 
	The total number of copies of  $Q_d$ in $Q_{d+2}$ is $4\binom{d+2}{2}$. \\

 Consider a set $A$ of vertices in $Q_{d+2}$ and let $M=M_A$ be an $|A|\times (d+2)$ binary matrix whose rows are the elements of $A$.   
	We shall call a copy of  $Q_d$ {\it good } if it contains exactly $s$ vertices from $A$, and call it {\it bad} otherwise.
 For each  pair $i,j$, $1\leq i<j\leq  d+2$ let  $M(i,j)$ be the $|A|\times 2$ sub-matrix of $M$ whole columns are  columns $i$ and $j$ of $M$.   A copy $Q$ of $Q_d$  with fixed positions $i$ and $j$ is good if and only if  there are exactly $s$ rows of $M(i,j)$ that match the values in the two  fixed positions of $Q$.\\

Thus $M(i,j)$ contributes to four (out of possible four)  good $Q_d$'s  with fixed positions $i,j$ if and only if 
$M(i,j)$ has exactly $s$ rows equal to each possible binary vector of length two. In particular, if $M$ contributes to four good $Q_d$'s, then $M$ has $4s$ rows and $|A|=4s$.
Otherwise $M(i,j)$ contributes to at most $3$ good $Q_d$'s.\\

Case 1.  $|A|\neq 4s$.  \\
	By the previous paragraph in this case 
	$\lambda(d+2, d, s, A) \leq 3/4$.\\

Case 2.  $|A|=4s$.  \\
If $M$ has a column $i$ with number of zeros not equal to $2s$, then $M(i,j)$  contributes at most $2$ good $Q_d$'s, for any $j\neq i$.  Let columns $i$ and $j$  have exactly $2s$ zeros each. Then $M(i,j)$ contributes $4$ good $Q_d$'s if and only if  these columns have exactly $s$ positions in which  both of them are zero, i.e., correspond to an edge of the 
	generalized Johnson's graph $J(4s, 2s, s)$.  Otherwise $M(i,j)$  contributes  no good $Q_d$'s.\\

Consider a complete  edge-weighted graph $G$ with vertex set $[d+2]$, whose vertices correspond to columns of $M$ and edges get a weight corresponding to the number of good $Q_d$'s contributed by the respective pairs of columns.
Thus $\lambda(d+2, d, s, A)$ is at most the total weight of $G$ divided by $4\binom{d+2}{2}$.\\

Assume that there are $k$ columns  of $M$  with exactly $2s$ zeros each.
Without loss of generality these are the first $k$ columns.  We see that the edges  of weight $4$ in $G$ correspond to a blow-up of a subgraph of  $J(4s, 2s, s)$. 
Since $\omega(J(4s,2s,s)=\omega(s)$, the total weight of edges 
contributed by the first $k$ vertices of $G$ is at most 
$4t(k, \omega(s))$.  All edges incident to $[d+2] \setminus [k]$ have weight at most $2$. Since the density of any non-trivial Tur\'an graph is at least $1/2$, the average weight of an edge induced by $[k]$ is at least $2$. Thus increasing $k$ does not decrease the total weight of $G$, that stays at most $4t(d+2, \omega(s))$. Note that this value is attained by a 
matrix with  $d+2$ columns having $2s$ zeros each and  
corresponding to the vertices of a  clique in $J(4s, 2s,s)$, each 
repeated an almost equal number of times.\\

Combining Case 1 and Case 2, we have that 
$$\pi(d+2, \omega(s))\leq \lambda(d+2, d, s) \leq \max \{3/4,  \pi(d+2, \omega(s)\}.$$

Clearly $\omega(1)=3$ and $\omega(s)>3$ for $s\geq 2$. Thus in particular $\lambda(d+2, d,1) \leq 3/4$, for $d>6$.
Note that $\pi(d+2, \omega(s))>3/4$ if $s>1$ or if   ($s=1$ and $d< 6$).  
Moreover, when $A$ is the set of all vertices in  $Q_{d+2}$ of 
Hamming weight $d+1$ or $0$, we have  $\lambda(d+2, d, 1, A) = 3/4$. 
Thus $\lambda(d+2, d, 1) = 3/4$ for $d\geq 6$.
This concludes the proof.
\end{proof}

\vskip 0.5cm
	
\begin{proof}[Proof of Theorem \ref{lb}]
	In order to lower bound $\lambda(n,d,1)$, consider a random set 
	$A$ of vertices in $Q_n$ obtained by choosing each
	vertex randomly and independently  with probability $2^{-d}$. 
	The probability that a copy of $Q_d$ contains exactly one vertex 
	of $A$ is $2^d\cdot 2^{-d} \cdot (1- 2^{-d})^{2^d-1}$, 
	and the desired result follows by linearity of expectation.\\

	For proving a lower bound for $\lambda(n,d,s)$, let 
$B$ be a random $d$ by $n$ binary matrix whose columns are 
independent uniformly chosen random nonzero vectors in $\F_2^d$. Define a
coloring of the vectors in $\F_2^n$ viewed as the vertices of $Q_n$
by coloring each vector by its syndrome $Bx \in \F_2^d$. If a set
$S$ of $d$ columns of $B$ 
	forms a basis, then the $2^d$ vertices of each of the
$2^{n-d}$ copies of $Q_d$  with any fixed values outside the
	columns of $S$
get all the $2^d$ possible colors.
Therefore, for every fixed choice of $s$ of these colors,
the set $A$ of all vertices with these colors has exactly
$s$ vertices of each of the subcubes corresponding to such nonsingular
sets of columns $S$. The expected fraction of such sets $S$ is
$c_d$ of all ${n \choose d}$ $d$-tuples of columns, and therefore
there exists a choice of $B$ for which there are at least that many
sets $S$. This completes the proof of the second lower bound.\\

Note that the subcubes that
have exactly $s$ vertices of the set $A$ above 
	are determined by the sets of 
their free coordinates, and not by the values of the fixed coordinates.
This is a property that, while not needed here,
may be helpful for some further applications.\\

The proof  of the last part  is very similar to the one above.
Let $B$ be a random $d-k$ by $n$ binary matrix whose columns are 
independent uniformly chosen random vectors in $\F_2^{d-k}$. Define a
coloring of the vectors in $\F_2^n$ viewed as the vertices of $Q_n$
by coloring each vector by its syndrome $Bx \in \F_2^{d-k}$. If a set
$S$ of $d$ columns of $B$ spans $\F_2^{d-k}$, then the 
	$2^d$ vertices of each of the
$2^{n-d}$ $d$-subcubes corresponding to any fixed choice of the 
values of the coordinates not in $S$ get each of the $2^{d-k}$ 
possible colors exactly $2^k$ times.
Therefore, for every fixed choice of $j$ of these colors,
the set $A$ of all vertices with these colors has exactly
$2^k \cdot j=s$ vertices of each of the subcubes 
corresponding to such spanning
sets of columns $S$. The expected fraction of such sets $S$ is
$c(d,k)$ of all ${n \choose d}$ $d$-tuples of columns, and therefore
there exists a choice of $B$ for which there are at least that many
sets $S$. This completes the proof. 
\end{proof}	
	
We remark that for small values of $d$ the inequality 
$\lambda(n,d,s)\geq  c(d, k)$  can be significantly improved to $c^*(d,k)$ defined similarly 
by restricting the random columns of the matrix to nonzero vectors. 
For example, if $d=3$ and $s=2$, then $k=1$ and  the lower bound $c^*(3,1)$ is $8/9$, 
whereas $c(3,1)=21/32$. 
As $d$ tends to infinity these two lower bounds converge to  the same value.

\section{Number theoretic consequences}\label{sec:number}
Recall that a layer in $Q_n$ is a maximal set of vertices with the 
same Hamming weight.
Consider a layered set $A$ of vertices in $Q_n$, that is a set of vertices that contains either all or none of the vertices of each layer. 
Then the number of vertices of $A$ in each copy of $Q_d$ is a sum of binomial coefficients $\binom{d}{i}$ for some values of $i$.  
If all copies of $Q_d$ have the same number, $s$, of vertices from $A$, we have $\lambda(n, d, s, A)=1$ and our results provide some 
simple properties of binomial coefficients.
The proof of the following proposition does not 
involve hypercube statistics, 
we include it here as the argument is direct, short and simple.
Theorem \ref{sum-binom}  is a generalisation of 
Proposition \ref{sum-binom-even} and its proof does use 
hypercube statistics.

\begin{proposition}\label{sum-binom-even}
For integers $k, a, d$, with $0\leq a<k$, $2<k\leq d$, let 
$q(a,k,d)$ be the sum of the 
binomial coefficients $\binom{d}{i}$ over all $i$, $i\equiv a \pmod k$. 
Then, for   any such fixed $d$ and $k$,  
the $k$ numbers $q(a,k,d)$, $0\leq a<k$ are not all equal.
\end{proposition}

\begin{proof}
Let $w$ be a primitive root of unity of order $k$.
Then $$q(a,k,d) = \frac{1}{k} \sum_{i=0}^{k-1} w^{-ia}(1+w^i)^d,$$
see for example \cite{G, BCK}. Let $v$ be the vector of length $k$ with coordinates $(1+w^i)^d$, $i=0, \ldots, k-1$, 
and let $A$ be the $k\times k$  Fourier matrix
$(w^{-ia})_{0\leq i,a<k}$. If all the numbers $q(a, k, d)$ are equal then $Av$ is a multiple of the constant vector.
Since $A$ is a nonsingular matrix and $A$ times the vector $(1, 0, \ldots , 0)$ is the vector $(1, 1, \ldots, 1)$
this implies that $v$ must be a multiple of the vector $(1, 0, \ldots , 0)$. But this is not the case for
$k > 2$ (note that it is the case for $k = 2$). 
\end{proof}

\begin{theorem}\label{sum-binom}
Let d and k be positive integers, and let T be a subset of $\Z_k$ (the
integers modulo k). For each element a in $\Z_k$ define
$q(a)=q(d,k,a,T)=\sum_i \binom{d }{i}$ where $i$ ranges over all numbers between
$0$ and $d$ for which $(i+a) \mod k$ lies in $T$.
If all numbers $q(a)$ are equal then their common value is $0$, $2^{d-1}$, or $2^{d}$.
Moreover, the only possibilities are $T=\Z_k$ or its complement
$T=\emptyset$, or $k$ even and $T$ either all even residues modulo $k$
or its complement, i.e., all odd residues modulo $k$.
\end{theorem}

\begin{proof}
Consider a prime $p>max (2^d,k)$, denote the common value of $q(a)$ by $s$,
and choose a subset $A$ of $Q_p$ by including in it exactly all layers
$b$ so that $b \pmod k$ lies in $T$.  Then every subcube $Q_d$ in this $Q_p$ contains
exactly $s$ vertices of $A$, so $Q_p$ contains $s2^{p-d}$ vertices of $A$, and
	Theorem \ref{ub}  (and its proof) imply that $s$ is either $0$ or
$2^d$, or $2^{d-1}$.

To prove the ``moreover" part (which was the reason to take $p>k$,
not only $p>2^d$, we allow $k$ here to be larger than $2^d$),
note that if $s=0$ then $T$ is empty, if $s=2^d$ then $T=\Z_k$, so assume
$s=2^{d-1}$. Consider the set $A$ defined as before but now we take it
in a bigger cube $Q_n$ ($n\gg p$). Every copy of $Q_p$ in $Q_n$ contains now
$s$ times $2^{p-d}=2^{p-1}$ vertices of $A$, which is $1$ modulo $p$. Therefore,
for every $i$,
layer $i$ is contained in $A$ if and only if layer $i+p$ is not contained in $A$
(since exactly one of the two binomial coefficients $\binom{p}{0}=1$
and $\binom{p}{p}=1$ should contribute to the number of vertices in the
cube $Q_p$ that lies in layers $i,i+1,\ldots,i+p$).

This means that if  $a \in \Z_k$ lies in $T$ then $a+p \pmod k$ does not,
and $(a+2p) \pmod  k$ is again in $T$.  For odd $k$,  $2p$ is relatively
prime to $k$, so this will give that $T$ is either empty or $\Z_k$
(which is not the case we are considering). So $k$ is even and
a lies in $T$ iff $a+p$ does not. As $p$ is relatively prime to $k$
this shows that $T$ is either all even or all odd residues
modulo $k$, completing the proof.
\end{proof}

\section{Approximate hypercube statistics}\label{sec:approx}

Consider a positive integer $d$. We know exactly for what values 
of $s=s(d)$, $\lambda(d,s)=1$.  These are $s\in \{0, 2^{d}, 2^{d-1}\}$. 
We say that the real value $x\in (0,1)$ is {\it approximately good} if for 
any sufficiently large $d$ and every $n>d$, there is a subset $A$ 
of vertices in $Q_n$ such that each copy of $Q_d$ contains  
$x2^d(1 +o(1))$ elements of $A$, where the $o(1)$ tends to zero as 
$d$ tends to infinity.

\begin{theorem}
	Any fixed real number $x\in (0,1)$ is approximately good.
\end{theorem}

\begin{proof}
	Approximate $x$ by a rational $p/q$ so that $|x-p/q| =o(x)$.
	Note that for any fixed $q$, as $d$ tends to infinity
	$qe^{-d/(10q^2)} =o(x)$. 

Let $A$ be a subset of vertices of the cube $Q_n$ consisting of all 
	layers that modulo $q$ belong to some fixed set $P$ of
$p$ elements of $\mathbb{Z}_q$.
Define $q(a,q ,d)$ to be the sum of the binomial coefficients 
	$\binom{d}{i}$ over all $i$, $i\equiv a \pmod q$, as in 
	the previous section.
Then each copy of $Q_d$ has $ \sum_{y\in P}  q(a+y, q, d)$ elements of 
	$A$, for some integer $a$.

Since $q(a,q ,d) = \frac{1}{q} \sum_{i=0}^{q-1} w^{-ia}(1+w^i)^d,$ 
	for the primitive root of unity $w$ of order $q$,  
	separating the term $i=0$ and using the triangle inequality,  
	we have 
	$$
	|q(a,q, d) - \frac{1}{q} 2^d | \leq  (2- \frac{1}{4q^2})^d 
	\leq 2^d e^{-d/(10q^2)}.
	$$
	Note that 
	the constants $4$  and $10$ here are not optimal and we make
	no attempt to optimize them.

Therefore, for each copy $Q$ of $Q_d$
$$
	||A \cap Q| - \frac{p}{q}2^d| \leq q 2^d e^{-d/(10q^2)}
	\leq o(x2^d).
	$$
	Since by our choice of the approximation $p/q$
	$$|\frac{p}{q}2^d-x2^d| \leq o(x2^d)$$
	the desired result follows.
\end{proof}

\noindent
{\bf Remarks:}
\begin{itemize}
\item
For $x=1/3$ the set $A$ consisting of every third layer of $Q_n$ contains
either $\lfloor 2^d/3 \rfloor$ or $\lceil 2^d/3 \rceil$ points in each copy
of $Q_d$, showing that in this specific case the approximation obtained
		is as strong as possible. Note also that for the unique 
		odd value 
		$s \in \{\lfloor 2^d/3 \rfloor, \lceil 2^d/3 \rceil \}$
		this shows that $\lambda(d,s) \geq 2/3-o(1)$ 
		with the $o(1)$-term tending to $0$ as $d$ tends to
	infinity. This is a better lower bound than the one provided by
		Theorem \ref{lb} for this case.
	\item
As is the case with all the questions here, the behaviour is
very different than the one with the analogous questions about
edge statistics in graphs: trying to maximize the number of 
		induced subgraphs on $d$ vertices in a large graph 
		that span exactly (or approximately) $s$ edges. 
		Here, by Ramsey's theorem, there are 
		always such induced subgraphs that
span either $0$ or $\binom{d}{2}$ edges, so no nontrivial approximation
to $s$ is possible if we want it to hold for all induced subgraphs on $d$
vertices.
\end{itemize}

\section{Concluding remarks}\label{sec:conclusion}
We considered the hypercube statistics problem 
expressed in the numbers $\lambda(d,s)$. 
We proved for a given $d$ that $\lambda(d,s)=1$ iff $s\in \{0, 2^d, 2^{d-1}\}$ and that for other values of $s$,  $\lambda(d,s)$ is at most  $1- \Omega(1/s)$ as $d$ grows.
We also showed that for those $s$ that are divisible by a high 
power of $2$  the lower bound on $\lambda(d,s)$ is close to 
the above upper bound. 
The following question remains open.\\

\noindent
{\bf Question.} What is the infimum of  $\lambda(d,s)$ 
over all admissible values of $d$ and $s$? Is it $c_d(1+o(1))$ for
large $d$? \\

By the probabilistic argument described in the second section, 
we know that $\lambda(d,s)$ is at least $c_d$, which is 
larger than $0.28$ for all $d,s$. 
However, we lack comparable upper bounds.
In particular we suspect that for large $d$,
$\lambda(d,1)=(1+o(1))1/e \approx 0.37 $, where the 
$o(1)$-term tends to $0$ as $d$ tends to infinity, but can only prove
a weaker upper bound arising from a tight bound for $\lambda(d+2,d,1)$. 
By a slightly more careful analysis we can show that for every fixed 
$d$, $\lambda(d,1)$ is strictly less than $3/4$, but this is still 
far from our best lower bound which approached $1/e$ as $d$ grows. \\

Our proof of the general upper bound involved a careful analysis of the $(d+2)$-cubes and averaging. We observe that the upper bound in $(d+2)$-cubes for $s>1$  is achieved by configurations with exactly $4s$ vertices of $A$.  One could then upper bound the fraction of 
$(d+2)$-cubes containing exactly $4s$ elements, and possibly also
iterate the argument. However, 
this approach only gives a modest improvement of the upper bound.\\

We showed that if $\lambda(d,s)=1$ then $s\in \{0, 2^d, 2^{d-1}\}$. 
If $s=0$ there is a unique set $A$ in $Q_n$ such that $\lambda(n, d, s, A)=1$, namely the empty set.
Similarly, for $s=2^d$, the only such set $A$ is the set of all vertices of $Q_n$. 
If $s=2^{d-1}$ and $d=1$, there are two possible sets $A$ such that $\lambda(n, d, s, A)=1$, the one consisting of all vertices of 
even Hamming weight and the one consisting of all vertices of odd 
Hamming weight.
If $s=2^{d-1}$ and $d>1$, there are more than two  such sets. Indeed, 
one can start with the set $A$ consisting of all vertices of even 
Hamming weight which satisfies 
$\lambda(n,d, s, A)=1$. Next, consider a $(n-d+t)$-subcube $Q$, 
for some $t\in [d-1]$,  and replace $A$ with its complement in 
this subcube. Let $B$ be the resulting set of vertices. 
For any $d$-cube $Q'$, $Q\cap Q'$ is a subcube of dimension at least $t$. 
Since any subcube of dimension at least one has exactly half of 
its vertices in $A$, it follows that the number of vertices of $B$ 
in $Q'$ is still exactly $2^{d-1}$. With certain restrictions, 
this process can be repeated to get additional sets $A$ that work.\\
	
When $s=2^d$, it is clear that  $\lambda(d,s)=1$ by taking  all the vertices in a ground hypercube. However, the problem of finding 
the largest possible value of 
$\lambda(n,d, 2^d, A)$  becomes non-trivial if we restrict 
the setting to the case when the size $k$ of $A$ is prescribed.
This problem is a  generalisation of the classical isoperimetric problem originally considered for $d=1$ that counts the largest number of edges induced by $k$ vertices in $Q_n$.  It was solved by Hart \cite{H}, as well as 
by quite a few others.
 Hardstun, Kratochv\'il, Sunde, and Telle \cite{SHKT}, see also Simon \cite{S} and Bollob\'as and Leader \cite{BL},  extended this problem to general $d$ and proved that for $|A|=k$,  
 $\lambda(n, d,  2^d,  A)$ is maximised by the set $A$ of $k$ 
 binary vectors that represent the first $k$ non-negative integers.
\\

\noindent	
{\bf Acknowledgements} The research of the first author is 
partially supported by NSF grant DMS-2154082. 
The research of the second author 
is funded in part by the DFG grant FKZ AX 93/2-1.


\end{document}